\documentclass[12pt]{amsart}
\usepackage[latin1]{inputenc}
\usepackage{amsmath,amsthm,amssymb}
\usepackage{mathrsfs}
\usepackage{color}
\usepackage[colorlinks]{hyperref}
\usepackage{tikz-cd}

\usepackage{enumerate}

\usepackage{algorithmic}
\usepackage[ruled]{algorithm}
\usepackage{caption}

\newtheorem{theorem}{Theorem}[section]
\newtheorem{lemma}[theorem]{Lemma}
\newtheorem{proposition}[theorem]{Proposition}

\theoremstyle{definition}
\newtheorem{definition}[theorem]{Definition}

\theoremstyle{remark}
\newtheorem{remark}[theorem]{Remark}

\newtheorem{notations}[theorem]{Notations}

\newcommand{\ic}{\ensuremath{\mathcal{I}}}
\newcommand{\oc}{\ensuremath{\mathcal{O}}}

\newcommand{\nc}{\ensuremath{\mathcal{N}}}

\newcommand{\Pq}{\mathbb{P}^4}

\newcommand{\Psx}{\mathbb{P}^6}

\newcommand{\Pn}{\mathbb{P}^n}

\newcommand{\Ps}{\mathbb{P}}

\newcommand{\bC}{\mathbb{C}}

\newcommand{\bK}{\mathbb{K}}

\DeclareMathOperator{\Proj}{Proj} 
\DeclareMathOperator{\Ext}{Ext} 

\DeclareMathOperator{\Hom}{Hom}

\DeclareMathOperator{\Ann}{Ann}
\DeclareMathOperator{\Ass}{Ass}

\def\bin #1#2 {\left( \matrix { #1 \cr #2 \cr } \right) }

\begin{document}

\title[ self-duality of the local cohomology  and Gherardelli's Theorem]
{ self-duality of the local cohomology of the Jacobian ring and Gherardelli's Theorem}

%    Information for  author
\author{Davide Franco }
\address{Universit\`a di Napoli
\lq\lq Federico II\rq\rq, 
Dipartimento di Matematica e
Applicazioni \lq\lq R. Caccioppoli\rq\rq,  Via
Cintia, 80126 Napoli, Italy} \email{davide.franco@unina.it}

\abstract  We prove that the $0$-th local cohomology of the jacobian ring of a projective hypersurface with isolated singularities has
a nice interpretation  it in the context of linkage theory. Roughly speaking, it represents a measure of the failure of Gherardelli's theorem for the corresponding graded modules.
This leads us to a different and characteristic free proof of its self-duality, which turns out to be an easy consequence of Grothendieck's local  duality theorem.

\bigskip\noindent {\it{Keywords}}: Linkage, Gherardelli's theorem, Jacobian ring, local cohomology, Grothendieck's local  duality theorem

\medskip\noindent {\it{MSC2010}}\,:   14B15;  14B05; 14H50

\endabstract
\maketitle

\bigskip
\section{Introduction}
\bigskip

In this paper we focus on the self-duality of the $0$-th local cohomology of the jacobian ring of a projective hypersurface with isolated singularities. Our main aim is to interpret it in the context of linkage theory.
Consider the polynomial ring  $ \bC [x_0, \dots , x_n]$  in n+1 variables, $n\geq 2$, with coefficients in the complex field $\bC$. Given a reduced polynomial $f\in \bC [x_0, \dots , x_n]$, homogeneous of degree $d$, the jacobian ring of $f$ 
is defined as
$$
 R=R(f):= \bC [x_0, \dots , x_n]\slash J(f),
$$
where 
$J(f)$ is the gradient ideal of $f$.
If $X=V(f)$, the hypersurface  defined by $f$, is nonsingular then the gradient ideal $J(f)$ is generated by a regular sequence hence $R$ is a \textit{Gorenstein artinian ring} with socle degree
$\sigma:= (n+1)(d-2)$. In particular, the main feature of the jacobian ring of a nonsingular hypersurface is 
Macaulay's theorem, that is to say the self-duality of $R(f)$:
\begin{equation}\label{Macaulay}
R(f)_k \cong R(f)_{\sigma - k }^{\vee}.
\end{equation}
As highlighted by the fundamental work of Griffiths and his school, $R$ encodes many properties of the geometry and the cohomology of $X$ as well as of his period map (cf. \cite{Griffiths}, \cite{Transc} and references therein).
For instance, in the seminal paper \cite{Griffiths}, Griffiths proved that  the jacobian ring determines the Hodge decomposition of the primitive cohomology in middle dimension.

Assume now that the hypersurface defined by $f$ is singular but reduced. In this case the jacobian ring $R(f)$ is no longer of finite length.
In the beautiful paper \cite{Sernesi}, E. Sernesi observed that \textit{most properties of the jacobian ring $R(f)$ in the nonsingular case are transferred to the local cohomology module $H^0_{\bf{m}}(R(f)) $ if $X$ is a singular hypersurface.}
The main ingredient in the study of the local cohomology module $H^0_{\bf{m}}(R(f)) $ is its interpretation by means logarithmic vector fields along  $X$. Specifically, if one denotes by $T_{\Pn} $ the tangent sheaf of $\Pn$, then one can define a subsheaf 
$T\langle X\rangle $ of $T_{\Pn} $, whose sections are vector fields on $\Pn$ with logarithmic behavior along $X$. There is an identification \cite{Sernesi}[Proposition 2.1]:
$$H^0_{\bf{m}}(R(f)) \cong H^1_{*}(T\langle X\rangle(-d)).$$

One of the main results of \cite{Sernesi}, is the self-duality of the above module $H^0_{\bf{m}}(R(f))$ when the hypersurface $X$ has \textit{only isolated singularities} \cite{Sernesi}[Theorem 3.4]:
\begin{equation}\label{self}
H^0_{\bf{m}}(R(f))_k \cong H^0_{\bf{m}}(R(f))_{\sigma - k}^{\vee },
\end{equation}
which is a sort of generalization of Macaulay's Theorem (\ref{Macaulay}).
The isomorphism above is a consequence of the following more general result (compare with \cite{Sernesi}[Theorem 3.2]):

\begin{theorem}[Sernesi]
\label{Sernesi} Let $I=(f_0, \dots , f_n)\subset \bC [x_0, \dots , x_n]$ be an ideal generated by $n+1$ homogeneous polynomials, of degrees $d_0, \dots , d_n$. Assume that $\dim (\Proj( \bC [x_0, \dots , x_n]\slash I))\leq 0$.
Then there is a natural isomorphism:
$$H^0_{\bf{m}}(\bC [x_0, \dots , x_n]\slash I) \cong [H^0_{\bf{m}}(\bC [x_0, \dots , x_n]\slash I)(\sigma)]^{\vee },
$$ where $\sigma:=\sum_0^n d_i -n-1$.
\end{theorem}
The proof of (\ref{Sernesi})  is an adaptation of the proof of Macaulay's Theorem, and rests on a spectral sequence argument for the hypercohomology of the Koszul complex.

The main aim of this paper is to point out that there is still another interpretation of the local cohomology ring, in the context of \textit{linkage theory} (we keep on the assumpion that $X$ has only isolated singularities). 
Roughly speaking,  \textit{ it measures the failure of Gherardelli's theorem to hold true for the corresponding graded modules} (cf. Theorem \ref{main} and Remark \ref{final}).
This leads us to a different and characteristic free proof of (\ref{Sernesi}), which turns out to be an easy consequence of \textit{Grothendieck's local  duality theorem} \cite[Theorem 3.6.19]{BH} \cite[Theorem A.1.9]{Eisenbud2}.

Since the appearance of the seminal paper of Peskine and Szpiro \cite{PS}, linkage theory turned out to be a fundamental tool in projective geometry. For instance, linkage theory  is a crucial ingredient in the study of
space curves (see e.g. \cite{GP}, \cite{Rao}, \cite{RaoII}, \cite{MDPerrin} and references therein), of surfaces in $\mathbb P^4$ \cite{PR} and, more generally, of projective varieties of codimension two  \cite{JAG},
\cite{EFGI}, \cite{EFGII}. More recently, combining linkage theory with a series of recent results about Noether-Lefschetz theory for linear systems with base locus  (\cite{IJM}, \cite{RCMP},  and \cite{CCMII}), it has been possible to prove a version of the  Speciality Theorem for curves contained in suitable singular
factorial threefolds of $\mathbb P^5$ (compare with  \cite{SpRicerche}[Theorem 3.2]).

We recall that a pair of algebraic schemes 
$\Delta$ and $\Theta$  (both contained in the same projective space $\Ps ^r$) are said \textit{algebraically linked} by a complete intersection $\Gamma\supset \Delta\cup \Theta$
if: 

a) $\Delta$ and $\Theta$ are equidimensional without embedded components;

b) $\ic_{\Delta\slash \Gamma}\cong \Hom( \oc _{\Theta}, \oc _{\Gamma})$ and $\ic_{\Theta\slash \Gamma}\cong \Hom( \oc _{\Delta}, \oc _{\Gamma})$.
\par\noindent
Assume that $\Gamma= V(f_1 , \dots ,f_n)$ and that $\Delta=V(f_0 , \dots , f_n)$, with $f_0, \dots f_n$ homogeneous polynomials  of degrees $d_0 , \dots , d_n$.
One of the key results of linkage theory is \textit{Gerardelli's theorem} (cf. \cite{FKL}[Theorem 2.5]:
\begin{equation}\label{Gerardelli}
\omega _{\Theta }\cong \oc _{\Theta }(\tau-d_0),  \,\,\, \tau:=\sum_1^n d_i -r-1.
\end{equation}

The main result of this paper is the remark  that, in some sense, \textit{the 0-th local cohomology module measures the failure of (\ref{Gerardelli}) to hold true for the corresponding graded modules} (cf. Theorem \ref{main} and Remark \ref{final}).
As mentioned before, this leads us to a  different and characteristic independent proof of (\ref{Sernesi}), which turns out to be an easy consequence of Grothendieck's local  duality theorem.

The proofs of our results can be found in \S 4. Sections 2 and 3, containing standard results of linkage theory which are  probably well-known to experts, are added in the attempt of making this paper reasonably self-contained. 
They consist in an adaptation to graded rings of some results usually stated for either local rings or ideal sheaves.

\medskip

\section{Preliminary results}

\medskip
Since the appearance of the seminal paper of Peskine and Szpiro \cite{PS}, linkage theory is mainly applied to local rings and to ideal sheaves of the projective space. One of the most popular results of
linkage theory is  Gherardelli's theorem, concerning the linkage of quasi-complete intersections projective schemes \cite{FKL}[Theorem 2.5]. As far as we know, Gherardelli's approach is much rarely
applied to graded rings.
In this section we collect some definitions and results concerning linkage theory for graded rings that are needed in the sequel. We also include  a proof of some results, although they are probably well-known, in the attempt of making this paper reasonably self-contained. 
They consist in an  adaptation to graded rings of some results usually stated for either local rings or ideal sheaves.
Since we are mainly interested to graded rings of dimension one, all  arguments are fairly plain. We follow very closely the approach of \cite[\S 21]{Eisenbud}.

\medskip 

\begin{notations}\label{dualmodule}
\begin{enumerate}
\item In this paper, $\bK$ denotes an algebraically closed field of any characteristic.
\item For any graded $\bK$-algebra $L=\bigoplus _0^{\infty}L_n$, with $L_0=\bK$, we denote by 
$L_+=\bigoplus _1^{\infty}L_n$ the \textit{irrelevant maximal ideal} of $L$.
\item Let $S:= \bK [x_0, \dots , x_n]$ be the polynomial ring in n+1 variables, with the usual grading.
Let $f_1, \dots , f_n$  be  a regular sequence of  homogeneous  polynomials of degrees $d_1, \dots , d_n$  and set $J:=(f_1, \dots , f_n)$. We denote by $T:=S \slash J$. Since $f_1, \dots , f_n$  is a regular sequence, $T$ is a \textit{complete intersection graded ring} of dimension $1$.
\item As usual, we denote by $\Omega _T$ the dualizing module of $T$. Since $T$ is a complete intersection  ring, we have
$\Omega _T \cong T[\tau]$, with $\tau:=\sum _{i=1}^n d_i -n -1$ \cite[\S 21.11]{Eisenbud}.
\end{enumerate}
\end{notations}

\begin{lemma}
\label{vanishing} Let $M$ be a $T$-module.
\begin{enumerate}
\item For $i> 1$, $\Ext _T^i(M, T)=0$.
\item If $M$ is a maximal Cohen-Macaulay  module \cite[\S 21.4]{Eisenbud}, then $\Ext _T^1(M, T)=0$.
\end{enumerate}
\end{lemma}
\begin{proof}
(1)  is a direct consequence of the fact that $T$ has injective dimension one, as follows from 
\cite[Theorem 3.1.17 and Proposition 3.6.11]{BH}.

As for (2), since $M$ is a maximal Cohen-Macaulay  module and $T$ has dimension one, there exists an element $x\in S_1$ which is
nonzerodivisor in $M$. Applying $\Hom _T(\cdot , T)$ to the short exact sequence
$$0 \rightarrow M \stackrel{x}{\rightarrow} M \rightarrow M\slash xM \rightarrow 0$$
we derive the long exact sequence
$$ \dots \rightarrow \Ext _T^1(M, T) \stackrel{x}{\rightarrow} \Ext _T^1(M, T) \rightarrow \Ext _T^2(M\slash xM,  T) \rightarrow \dots .$$
By point (1), we have $\Ext _T^2(M\slash xM,  T)=0$, hence we are done owing to Nakayama's Lemma [Lemma 1.4]\cite{Eisenbud2}.
\end{proof}
\medskip

\begin{lemma}\label{nakayama}
Let $A$ be a graded $\bK$-algebra, and let $M$ and $N$ be finitely generated graded $A$-modules. Suppose that $x\in A_+$ is a nonzerodivisor on $N$. Then a morphism of graded modules $\phi: M\to N$ 
is an isomorphism if and only  the induced map  $\psi : M\slash xM \to N\slash xN$ is  an isomorphism.
\end{lemma}
\begin{proof}
The surjectivity is an immediate consequence of Nakayama's Lemma [Lemma 1.4]\cite{Eisenbud2}.

As for the injectivity, set $K:= \ker \phi $ and assume $\psi$ is an isomorphism. Of course 
we must have $K\subset xM$.   
Furthermore, since $x$ is a nonzerodivisor on $N$, we have $(K:_M x)=K$. 
Then we have
$$m\in K \,\, \Rightarrow \,\,  m=xm' \,\, \Rightarrow \,\,  m'\in (K:_M x)=K \,\,  \Rightarrow \,\, m\in xK, $$
where the first implication comes from the inclusion $K\subset xM$.
Then $xK=K$ and conclude again by Nakayama's Lemma.
\end{proof}

The following result is an adaptation to our context of  the Theorem 21.21 of \cite{Eisenbud}:

\begin{theorem}[Duality] 
\label{duality}
Let $D$ be the functor $\Hom _T(\cdot , \Omega _T)$. Then $D$ is a dualizing functor 
on the category of maximal Cohen-Macaulay $T$-modules in the sense that
\begin{enumerate}
\item $D$ takes maximal Cohen-Macaulay $T$-modules to maximal Cohen-Macaulay $T$-modules.
More precisely, if $M$ is a maximal Cohen-Macaulay $T$-module then $\Hom _T(M , \Omega _T)\not=0$. Furthermore, if $x$ is a nonzerodivisor on $M$ and  $T$, then  $x$ is a nonzerodivisor on $\Hom _T(M , \Omega _T)$
and 
$$\Hom _T(M , \Omega _T)\slash x \Hom _T(M , \Omega _T)\cong \Hom _{T\slash xT}(M\slash xM  , \Omega _T \slash x\Omega _T).
$$
\item $D$ takes exact sequences of maximal Cohen-Macaulay $T$-modules  to exact sequences.
\item The natural map $M\to \Hom _T(\Hom _T(M , \Omega _T) , \Omega _T)$ is an isomorphism when $M$ is 
a maximal Cohen-Macaulay $T$-module.
\end{enumerate}
\end{theorem}
\begin{proof}
(1). 
If $\Hom _T(M , \Omega _T)$  vanished,  then  $\Ann M$ would contain a $T$-regular element (cf. \cite[Proposition 1.2.3]{BH}). Then, the module $M$ would be supported in codimension one, in contrast with the assumption it is a maximal Cohen-Macaulay $T$-module.
Assume now
$x$ is a nonzerodivisor on $M$ and  $T$,  which entails $\Hom _T(T\slash xT , \Omega_T)=\Hom _T(T\slash xT , T[\tau])=0$ (cf. \cite[Proposition 1.2.3]{BH}).
Applying the functor $\Hom _T(M[-\tau] , \cdot)$ to the short exact sequence
$$0 \rightarrow T \stackrel{x}{\rightarrow} T \rightarrow T\slash xT \rightarrow 0$$
we derive a long exact sequence beginning with
\begin{equation}\label{longes}
0 \rightarrow \Hom _T(M , \Omega_T) \stackrel{x}{\rightarrow} \Hom _T(M , \Omega_T) \rightarrow \Hom _T(M , \Omega_T\slash x \Omega_T) \rightarrow 
\Ext ^1 _T(M , \Omega_T)=0,
\end{equation}
(the last equality follows from Lemma \ref{vanishing} (2)).
Thus $x$ is a nonzerodivisor on $\Hom _T(M , \Omega _T)$. On the other hand,
any homomorphism $M\to \Omega_T\slash x \Omega_T$ factors uniquely through $M\slash xM$, hence
$$\Hom _T(M , \Omega_T\slash x \Omega_T)\cong \Hom _T(M\slash xM , \Omega_T\slash x \Omega_T) \cong \Hom _{T\slash xT}(M\slash xM  , \Omega _T \slash x\Omega _T),$$
and (\ref{longes}) becomes 
$$
0 \rightarrow \Hom _T(M , \Omega_T) \stackrel{x}{\rightarrow} \Hom _T(M , \Omega_T) \rightarrow \Hom _{T\slash xT}(M\slash xM  , \Omega _T \slash x\Omega _T) \rightarrow 
0.
$$
(2). If $0\to M' \to M \to M'' \to 0$ is a short exact sequence of $T$-modules, with $M''$ maximal Cohen-Macaulay, then just applying 
$\Hom _T(\cdot , \Omega _T)$ we get
$$
0 \rightarrow \Hom _T(M'' , \Omega_T) \rightarrow \Hom _T(M , \Omega_T) \rightarrow \Hom _T(M' , \Omega_T) \rightarrow 
\Ext ^1 _T(M'' , \Omega_T)=0,
$$
where the last equality follows from Lemma \ref{vanishing} (2).

(3). Consider a general $x\in S_1$. Then $x, f_1, \dots , f_n$ is a regular sequence in $S$ and in view of \cite[\S 21.11]{Eisenbud} we have
$$
\Omega _T \slash x\Omega _T \cong T \slash xT [\tau] \cong \Omega _{T\slash xT}[-1].
$$
On the other hand, $T\slash xT$ is an artinian ring and $D'(\cdot ):= Hom_{T\slash xT} (\cdot ,\Omega _{T\slash xT})$ is a dualizing functor
on the category of $T\slash xT$-modules \cite[\S 21.1]{Eisenbud}.
Denote by $\phi _M: M\to D(D(M))$ the natural map. By point (1), we have 
$D(M)\slash xD(M) \cong D'(M\slash xM[1])$, hence
$$
D(D(M))\slash xD(D(M))\cong D'( D(M)\slash xD(M)[1])\cong D'(D'(M\slash xM)).
$$
We conclude just applying Lemma \ref{nakayama} to the  following commutative diagram

\begin{equation}
\begin{array}{ccccccccc}
0 & \rightarrow & M & \stackrel{x}{\rightarrow} & M & \rightarrow & M\slash xM & \rightarrow &  0 \\
  &             &\stackrel {\phi _M}{}\downarrow    &                        &  \stackrel {\phi _M}{}\downarrow      &                            &     \updownarrow     &             &           \\
0 & \rightarrow & D(D(M)) & \stackrel{x}{\rightarrow} & D(D(M)) & \rightarrow &D'(D'(M\slash xM)) & \rightarrow &  0. \\
\end{array}
\end{equation}
\end{proof}

\bigskip

\section{Gherardelli Linkage}

\medskip
Let $I:=(f_k, \dots , f_n)\subset S$, $0\leq k < n$, be a proper  ideal, with homogeneous  generators of degrees $d_k, \dots , d_n$  and assume $d_k\leq d_{k+1} \leq \dots \leq d_n$.
Set $R:=S \slash I$ and assume moreover $\dim \Proj R=k$. We call such a ring  a \textit{quasi-complete intersection ring}.
Further, we denote by $\Delta:= \Proj R\subset \Pn $ the projective $k$-dimensional scheme associated to $R$, and by $\ic_{\Delta}:=\widetilde{I}\subset \oc_{\Ps ^n}$
its ideal sheaf. 

\begin{lemma}
\label{existsci} With notations as above, we can assume that $f_{k+1}, \dots , f_n$ is a regular sequence.
\end{lemma}
\begin{proof}
We prove by decreasing induction that, for any $i\geq k+1$, there exists a new system of homogeneous generators $(f_k, \dots , f_{i-1}, f_i', \dots , f_n')=I$, with $\deg f_i= \deg f_i'$ and such that $f_i', \dots , f_n'$ is a regular sequence.
Since the starting case $i=n$ is obvious, we assume by induction $i<n$ and $f_{i+1}', \dots , f_n'$ is a regular sequence. Denote by $\Gamma _i:= \Proj(S\slash (f_{i+1}', \dots , f_n'))\subset \Pn$ the complete intersection
projective scheme defined by the regular sequence  $f_{i+1}', \dots , f_n'$. Of course we have $\Delta \subset \Gamma _i$ and the ideal sheaf $\ic _{\Delta \slash \Gamma _i}$ is generated
by $f_{k}, \dots , f_i$. Fix a prime 
$\pi \in \Ass (S\slash (f_{i+1}', \dots , f_n'))$ and denote by  $\Gamma_{\pi}:= \Proj(S\slash \pi)$ the component of $\Gamma_i$ corresponding to $\pi$.
Since 
$\ic _{\Delta \slash \Gamma _{\pi}}$ is generated in degree $d_i$,
 the general element  $f_i'\in (f_{k}, \dots , f_i)_{d_i}$ does not vanish in $S\slash \pi$. As this holds true for any component of $\Gamma_i$, $f_i'$ is a nonzerodivisor in $S\slash (f_{i+1}', \dots , f_n')$. 
 We are done, in view of
$(f_{k}, \dots , f_i)=(f_{k}, \dots , f_{i-1}, f_i')$.
\end{proof}

\medskip

In the hypotheses of Lemma \ref{existsci}, we can assume that
the  ideal $J:=(f_{k+1}, \dots , f_n)$ is a \textit{complete intersection homogeneous ideal}. We denote by 
$\Gamma$ the corresponding complete intersection $k$-dimensional scheme, and by $\ic_{\Gamma}:=\widetilde{J}\subset \oc_{\Ps ^n}$
its ideal sheaf. 

\medskip
We borrow from \cite{PS} the main definition of linkage theory of projective varieties.

\begin{notations}
\label{deflinkageproj}
Let $\Delta$ and $\Theta$  be algebraic schemes contained in some projective space. We recall that the schemes 
$\Delta$ and $\Theta$ are said \textit{algebraically linked} by a complete intersection $\Gamma\supset \Delta\cup \Theta$
if: 

a) $\Delta$ and $\Theta$ are equidimensional without embedded components;

b) $\ic_{\Delta\slash \Gamma}\cong \Hom( \oc _{\Theta}, \oc _{\Gamma})$ and $\ic_{\Theta\slash \Gamma}\cong \Hom( \oc _{\Delta}, \oc _{\Gamma})$.
\end{notations}

\medskip

We quote some well known results from linkage theory (cf. \cite{PS}, \cite{FKL}) for which we sketch a short proof, for the benefit of the reader.
\begin{proposition}
\label{0dim} With notations as above, define the residual (of $\Delta$ in $\Gamma$) $k$-dimensional scheme $\Theta$  in such a way that $\ic _{\Theta  \slash \Gamma}:= \Ann_{\oc _{\Gamma}} \ic _{\Delta \slash \Gamma}$ and set $\lambda:=\sum_{k+1}^nd_i-n-1$. 
\begin{enumerate}
\item The dualizing sheaves and the ideals in $\Pn $ of $\Delta$ and $\Theta$ are related by the following short exact sequence
$$ 0 \rightarrow \ic _{\Gamma \slash \Pn}(\lambda ) \rightarrow \ic _{\Delta \slash \Pn}(\lambda )\rightarrow \omega _{\Theta }\rightarrow 0;
$$
\item (Gherardelli) we have an isomorphism $\omega _{\Theta }\cong \oc _{\Theta }(\lambda-d_k)$.
\end{enumerate}
\end{proposition}
\begin{proof} We follow \cite[Proposition 2.3 and Theorem 2.5]{FKL}.

(1). By duality theory we have
$$
\omega _{\Gamma}\cong \Ext ^n_{\Pn}(\oc_{\Gamma}, \oc_{\Pn}(-n-1))\cong \Hom (\det \nc ^*_{\Gamma \slash \Pn}, \oc_{\Gamma}(-n-1))\cong \oc_{\Gamma}(\lambda).
$$
By \cite[2.1.1]{FKL}, we have
$$
\omega _{\Theta}\cong \Hom (\oc_{\Theta}, \omega_{\Gamma})\cong \Hom (\oc_{\Theta}, \oc_{\Gamma})\otimes \omega_{\Gamma}\cong
\ic_{\Delta \slash \Gamma}\otimes \omega_{\Gamma}.
$$ 
To conclude, it suffices to consider the short exact sequence
$$
0\rightarrow \ic_{\Gamma \slash \Pn}\rightarrow \ic_{\Delta \slash \Pn}\rightarrow \ic_{\Delta \slash\Gamma }\rightarrow 0.
$$

(2). The ideal sheaf $\ic _{\Delta \slash \Gamma} $ is generated by $f_k$, hence multiplication 
by $f_k$ provides a surjection $\oc _{\Gamma}(-d_k) \to \ic _{\Delta \slash \Gamma}$. By \cite[2.1.1]{FKL}, its kernel is $\Ann \ic _{\Delta \slash \Gamma}[-d_k]\cong \ic _{\Theta \slash \Gamma}[-d_k]$, so we have an isomorphism
\begin{equation}\label{sheaftheoreticiso}
\oc _{\Theta}(-d_k) \cong  \ic _{\Delta \slash \Gamma}\cong \Hom_{\oc _{\Gamma}}(\oc _{\Delta}, \oc _{\Gamma})
\end{equation}
 and we are done, owing to (1).
\end{proof}

\medskip

\begin{remark}
\label{fruitful} When $k=0$, the meaning of Proposition \ref{0dim} (2) is that the residual $0$-dimensional scheme $\Theta$ is Gorenstein. In light of this, it may seem curious to state Gherardelli's isomorphism
in the form 
$\omega _{\Theta }\cong \oc _{\Theta }(\tau-d_0)$, since in this case we have $\omega _{\Theta }\cong \oc _{\Theta }(l)$, $\forall l$. Nevertheless, for future reference (cf. Remark \ref{final}) we find it convenient  to state the previous isomorphism in the given form.
\end{remark}

\bigskip

\section{Selfduality of the local cohomology}

\medskip
We now introduce a graded version of  \ref{deflinkageproj}.
\begin{definition}
\label{deflinkage}
Let $T$ be a complete intersection graded ring and let $H$ and $K$ be graded ideals of $T$, of height $0$. Set $A:=T\slash H $ and $B:=T\slash K$. The rings  
$A$ and $B$ are said \textit{algebraically linked by the complete intersection} $T$
if: 

a) $A$ and $B$ are equidimensional without embedded components;

b) $H\cong \Hom_T(B, T)$ and $K\cong \Hom_T(A, T)$.

\end{definition}

\medskip
 
Following \S 3,
let $I=(f_0, \dots , f_n)\subset S$ be an ideal generated by $n+1$ homogeneous polynomials, of degrees $d_0, \dots , d_n$, and set $R:=S\slash I$. Assume that $\dim \Proj R= 0$. Then $R$ is a quasi-complete intersection ring to which we  apply the Gherardelli linkage of \S 3. In particular, by Lemma \ref{existsci}  we may assume that the sequence  $f_1, \dots , f_n$ is regular. Hence the ideal
$J:=(f_1, \dots , f_n)$ provides a complete intersection scheme of dimension $0$ in $\Ps ^{n}$. Following the notations of previous sections, we set $T:= S\slash J$, $\Gamma:= \Proj T$ and $\Delta:= \Proj R$.

\begin{notations}
\label{CM} 
Let  $I^s:= \bigcup _{i=1}^{\infty}(I : S_+^i)$ be the saturation of $I$. Of course we have $\widetilde{I} = \widetilde{I^s}$.
Since $I^s$ is a saturated ideal of height n, the ring $A:=S\slash I^s$ is a one-dimensional  Cohen-Macaulay graded ring. Specifically, it is a maximal Cohen-Macaulay T-module
to which we are allowed to apply the results of \S 2.
For future reference, we recall the following relation involving the local cohomology of the $S$-module $R$ \cite[(5)]{Sernesi}:
\begin{equation}
\label{loccohom} H^0_{S_+}(R)\cong I^s \slash I.
\end{equation}
\end{notations}

\medskip

\begin{theorem}
\label{linkage} Denote by $H$ the image of $I^s$ in $T$ (so that we have $A=T\slash H$). Further,  set  $K:=(0:_T H)$ and $B:=T\slash K$. Then we have $K \cong\Hom _T(A, T)$ and the rings $A$ and $B$ are algebraically linked by the complete intersection $T$.
\end{theorem}
\begin{proof}
We follow \cite[Theorem 21.23]{Eisenbud} and \cite[Proposition 2.1]{FKL}.
Since $A= T\slash H$, the map $\Hom _T(A, T)\to T$ given by evaluating at 1 is an isomorphism onto $(0:_T H)$, proving the first statement.
By  Theorem \ref{duality} (1), $K$ is a one-dimensional  Cohen-Macaulay ring. By  Theorem \ref{duality} (2),
the short exact sequence
$$
0\rightarrow K \rightarrow T \rightarrow B \rightarrow 0
$$
is, up to a twist, the dual of
$$
0\rightarrow H  \rightarrow T \rightarrow A \rightarrow 0
$$
By  Theorem \ref{duality} (3), the dual of the former short exact sequence is the latter one, hence we have $H  \cong\Hom _T(B, T)$.
Finally,   \cite[Corollary 18.6]{Eisenbud} entails that both $H$ and $K$ are one-dimensional  Cohen-Macaulay rings and so are $A$ and $B$, owing to  Theorem \ref{duality} (1).
\end{proof}
 
\medskip

\begin{remark}
\label{alsoIfrac} 
\begin{enumerate}
\item Observe that Theorem \ref{linkage} implies that $H   \cong\Hom _T(B, T) \cong (0:_T K)$.
\item With notations as in Proposition \ref{0dim} and Theorem \ref{linkage}, it is fairly standard to prove that $\Theta= \Proj B$ and $\ic _{\Theta \slash \Gamma}= \widetilde{K}$ 
\end{enumerate}
\end{remark}

\medskip

\begin{definition}
\label{matlis}
For any graded Cohen-Macaulay $\bK$-algebra $T=\bigoplus _0^{\infty}T_n$ of dimension $d$ and for any finitely generated graded $T$-module 
$M$ we denote by
$$
M^{\vee}=\bigoplus _k (M^{\vee})_k=\bigoplus _k \Hom_{\bK}(M_{-k}, \bK)
$$
the  \textit{Matlis dual} of $M$ \cite[p. 141]{BH}.

Recall the fundamental \textit{Grothendieck's local  duality theorem} \cite[Theorem 3.6.19]{BH} \cite[Theorem A.1.9]{Eisenbud2}: for all finitely generated graded $T$-modules $M$ and all integers i there exist natural homogeneous isomorphisms
\begin{equation}
\label{Grothduality} (H^i_{T_+}(M))^{\vee}\cong \Ext ^{d-i}_T(M, \Omega_T).
\end{equation}
\end{definition}

\bigskip
We are now ready to prove the main result of this paper:
\begin{theorem}
\label{main} Keep Notations as in (\ref{linkage}). We have a short exact sequence of graded modules:
$$
0 \rightarrow B[\tau -d_0] \rightarrow \Hom _T(B, \Omega _T) \rightarrow H^0_{T_+}(R)[\tau] \rightarrow 0,
$$
with $\tau:=\sum _{i=1}^n d_i -n -1$.
Moreover, the local cohomology module $ H^0_{T_+}(R)$ is self-dual 
$$(H^0_{T_+}(R)[\tau +d_0])^{\vee}\cong H^0_{T_+}(R).$$
\end{theorem}
\begin{proof}
First of all we define a morphism $B[-d_0] \rightarrow \Hom _T(B, T)$. Recall that  $f_0\in I\subset I^s$ and $H \cong \Ann _T K$ is the image of $I^s$ in $T$ (cf. Theorem \ref{linkage} and Remark \ref{alsoIfrac}). Then, the multiplication by $f_0$
provides a map 
\begin{equation}\label{noquotient}
 S[-d_0]\stackrel{ f_0}{\longrightarrow} \Hom _T(B, T).
\end{equation}
 On the other hand,  $K \cong \Ann _T H$ implies that the previous morphism
factors through
\begin{equation}
\label{injective} B[-d_0]  \stackrel{ f_0}{\longrightarrow} \Hom _T(B, T)
\end{equation}
We claim that such a map is injective. First of all, we observe that the sheaf-theoretic counterpart of the above morphism is the Gherardelli isomorphism (\ref{sheaftheoreticiso})
(cf. Remark \ref{alsoIfrac} (2)).
Hence, the kernel of (\ref{injective}) must be supported on the maximal graded ideal. On the other hand, Theorem \ref{linkage} entails that $B$ is a
maximal Cohen-Macaulay $T$-module. So $T$ cannot have submodules supported on the maximal  ideal and the Claim is so proved.

Thus we have an injective morphism $B[\tau -d_0] \rightarrow \Hom _T(B, \Omega _T)$ and we denote by $N$ its cokernel. We need to prove that $N$ 
is, up to an appropriate shift, the local cohomology module $H^0_{T_+}(R)$. First of all, we observe that (\ref{noquotient}) and (\ref{injective}) have the same cokernel $N[-\tau]$.
Hence, by taking into account Remark  \ref{alsoIfrac} we have an exact sequence
\begin{equation}\label{longsurj}
S[-d_0]\stackrel{ f_0}{\longrightarrow}  H \longrightarrow N[-\tau] \longrightarrow 0.
\end{equation}
Since $H $ is the image of $I^s$ in $T=S\slash J$, with $J=(f_1, \dots ,f_n)$, we  have also
$$
\bigoplus_0^n S[-d_i]\stackrel{ (f_0, \dots , f_n)}{\longrightarrow}  I^s \longrightarrow N[-\tau] \longrightarrow 0
$$
and we get $N[-\tau]\cong I^s \slash I\cong H^0_{T_+}(R)$, owing to  (\ref{loccohom}).

In order to conclude, we only need to prove that the local cohomology is self-dual. Obviously, the local cohomology module $H^0_{T_+}(R)$ is supported on the maximal ideal
hence $\Hom_T(H^0_{T_+}(R), \Omega _T)=0$ (cf. \cite[Proposition 1.2.3]{BH}).
Applying 
$\Hom _T(\cdot , \Omega _T)$ to the short exact sequence
\begin{equation}
\label{tocompare}
0 \rightarrow B[\tau -d_0] \rightarrow \Hom _T(B, \Omega _T) \rightarrow H^0_{T_+}(R)[\tau] \rightarrow 0
\end{equation}
and taking into account of Lemma \ref{vanishing} (2),
we get 
$$
0 \rightarrow B \rightarrow \Hom _T(B[\tau -d_0], \Omega _T) \rightarrow \Ext ^1 _T (H^0_{T_+}(R)[\tau], \Omega _T) \rightarrow 0.
$$
By \cite[Theorem 3.2.13 and Theorem 3.5.8]{BH}, $H^0_{T_+}(R)$ is an Artinian module so
 \cite[Corollary A 1.5]{Eisenbud2} implies $H^0_{T_+}(H^0_{T_+}(R))=H^0_{T_+}(R)$.  
Then,  Grothendieck's duality theorem (\ref{Grothduality}) entails
$$
(H^i_{T_+}(H^0_{T_+}(R)))^{\vee}=(H^0_{T_+}(R))^{\vee}\cong \Ext ^{1}_T(H^0_{T_+}(R), \Omega_T),
$$
and we get
$$
0 \rightarrow B[\tau -d_0] \rightarrow \Hom _T(B, \Omega _T) \rightarrow 
(H^0_{T_+}(R)[\tau])^{\vee}[\tau -d_0] \rightarrow 0.
$$
By comparison of  (\ref{tocompare}) with the last short exact sequence we conclude that
$$H^0_{T_+}(R)[\tau]\cong (H^0_{T_+}(R)[\tau])^{\vee}[\tau -d_0]\cong (H^0_{T_+}(R)[\tau+d_0])^{\vee}[\tau]$$
and we are done.
\end{proof}

\medskip

\begin{remark}\label{final}
As already observed	in the proof of Theorem \ref{main}, Gherardelli's isomorphism $\omega _{\Theta }\cong \oc _{\Theta }(\tau-d_0)$ is the sheaf-theoretic counterpart of
the morphism $B[\tau -d_0] \hookrightarrow \Hom _T(B, \Omega _T)$. Thus, the first statement of the previous theorem
says that the local cohomology module $H^0_{T_+}(R)$
\textit{measures the failure of Gherardelli's theorem to hold true for the corresponding graded modules}. 
\end{remark}


\bigskip

\begin{thebibliography}{s=5}


\bigskip

\bibitem{BH} Bruns, - Herzog , D.: {\it Cohen-Macaulay rings}. Cambridge studies in Advanced Mathematics {\bf 39},   Cambridge Univ. Press (1998).



\bibitem{IJM} Di Gennaro, V. - Franco, D.: {\it Noether-Lefschetz Theory and N\'eron-Severi group}.
Int. J. Math. {\bf 23}, No 1, Article ID 1250004, 12 p. (2012). DOI: 10.1142/S0129167X11007483


\bibitem{RCMP} Di Gennaro, V. - Franco, D.: {\it Noether-Lefschetz theory with base locus}.
Rend. Circ. Mat.  Palermo (2) {\bf 63}, No 2, 257-276, (2014). DOI 10.1007/s12215-014-0156-8

\bibitem{CCMII} Di Gennaro, V. - Franco, D.: {\it N\'eron-Severi group of a general hypersurface},
Commun. Contemp. Math., Vol. 19, No. 01, 1650004 (2017).
DOI: 10.1142/S0219199716500048


\bibitem{SpRicerche} Di Gennaro, V. - Franco, D.: {\it A speciality theorem for curves in $\mathbb P ^5$
 contained in Noether-Lefschetz general fourfolds.} Ric. Mat. {\bf 66}, No 2, 509-520 (2017). DOI 10.1007/s11587-016-0316-6

\bibitem{Dimca2} Dimca, A.: {\it Sheaves in Topology}, Springer Universitext, (2004).

\bibitem{Eisenbud} Eisenbud, D.: {\it Commutative algebra. With a view toward algebraic geometry},
Graduate Texts in Mathematics 150. Springer Verlag (1995).

\bibitem{Eisenbud2} Eisenbud, D.: {\it The Geometry of Syzygies},
Graduate Texts in Mathematics 229. Springer Verlag (2005).

\bibitem{JAG} Ellia Ph. - Franco, D.: {\it On codimension two subvarieties of in $\mathbb P ^5$ and $\Psx$}. J. Algebraic Geom. \textbf{ 11}, No 3, 513-533 (2002). S 1056-3911(02)00320-X

\bibitem{EFGI} Ellia Ph. - Franco, D. - Gruson L.: {\it On subcanonical surfaces of $\Pq$}. Math. Z. {\bf 251}, No 2, 257-265 (2005). DOI: 10.1007/s00209-005-0796-7

\bibitem{EFGII} Ellia Ph. - Franco, D. - Gruson, L.: {\it Smooth divisors of projective hypersurfaces}. Comment. Math. Helv. {\bf 83},  371-385 (2008). DOI 10.4171/CMH/128


\bibitem{FKL} Franco, D. - Kleiman S. L. - Lascu, A.T.: {\it Gherardelli Linkage and Complete Intersections}. Mich. Math. J. {\bf 48}, Spec. Vol., 271-279, (2000).

\bibitem{Griffiths} Griffiths, P.:  {\it On the periods of certain rational integrals I, II}. Annals of Math. {\bf 90},  460-541, (1969).

\bibitem{Transc} Griffiths, P.:  {\it Topics in transcendental algebraic geometry}. Annals of Math. studies {\bf 106},   (1984).

\bibitem{GP} Gruson L. - Peskine, Ch.: {\it Genre des courbes de l'espace projectif.} Algebraic geometry, Springer LNM {\bf 687}, 31-59 (1978).


\bibitem{MDPerrin} Martin-Deschamps, M. - Perrin, D.: {\it Sur la classification des courbes gauches}, Ast\'erisque No 184-85, (1990).

\bibitem{PS} Peskine, Ch. - Szpiro, L.: {\it Liaison des vari\'et\'es alg\'ebriques}. Invent. Math. {\bf 26}, No 1, 271-302 (1974).

\bibitem{PR} Popescu, S. - Ranestad, K.: {\it Surfaces of degree 10 in the projective fourspace via linear systems and linkage}, J. Algebraic Geom. {\bf 5}, No 1, 13-76 (1996).

\bibitem{Rao} Rao. P.: {\it Liaison among curves in $\mathbb P^3$}, Invent. Math. {\bf 50}, No 3, 205-217 (1978/79).

\bibitem{RaoII} Rao. P.: {\it Liaison equivalence classes}, Math. Ann. {\bf 258}, No 2, 169-173 (1981/82).

\bibitem{Sernesi} Sernesi, E.: {\it The local cohomology of the jacobian ring}, Doc. Math.  {\bf 19}, 541-565 (2014).



\end{thebibliography}
\end{document}